\documentclass{amsart}
\usepackage{amsmath,amsfonts,amsthm,amssymb,amsrefs}
\usepackage{tikz}
\usetikzlibrary{arrows}
\usepackage{soul}

\newtheorem{theorem}{Theorem}
\newtheorem{cor}{Corollary}
\newtheorem{lemma}{Lemma}
\newtheorem{prop}{Proposition}

\begin{document}
\title{Length spectra and strata of flat metrics}
\date{}
\author{Ser-Wei Fu}

\begin{abstract}
In this paper we consider strata of flat metrics coming from quadratic differentials (semi-translation structures) on surfaces of finite type. We provide a necessary and sufficient condition for a set of simple closed curves to be spectrally rigid over a stratum with enough complexity, extending a result of Duchin-Leininger-Rafi. Specifically, for any stratum with more zeroes than the genus, the $\Sigma$-length-spectrum of a set of simple closed curves $\Sigma$ determines the flat metric in the stratum if and only if $\Sigma$ is dense in the projective measured foliation space. We also prove that flat metrics in any stratum are locally determined by the $\Sigma$-length-spectrum of a finite set of closed curves $\Sigma$.
\end{abstract}

\maketitle


\section{Introduction}\label{intro}

In their recent work \cite{DLR}, Duchin, Leininger, and Rafi studied the length spectral rigidity problem for the space of flat metrics coming from quadratic differentials on surfaces. They prove that a set $\Sigma$ of simple closed curves is spectrally rigid over $\mathsf{Flat}(S)$ if and only if $\overline{\Sigma}=\mathcal{PMF}(S)$; see Sections~2--4 for definitions.

In this paper, we have extended the results of \cite{DLR} to strata of flat metrics. Let $S$ be a surface of finite type with genus $g$ and $n$ marked points that satisfies $(3g+n-3)\geq 2$ throughout the paper. We prove that a set $\Sigma$ of simple closed curves is spectrally rigid over a stratum $\mathsf{Flat}(S,\alpha)$ with sufficiently high dimension if and only if $\overline{\Sigma} = \mathcal{PMF}(S)$. To state this more precisely, recall that a stratum is determined by $\alpha = (\alpha_1, \dots, \alpha_{k} ; \varepsilon)$, where $k=k_0+n$, $k_0$ is the number of zeroes at non-marked points, $\alpha_1, \ldots, \alpha_{k_0}$ are the order of the zeroes, $\alpha_{k_0+1}, \ldots, \alpha_{k}$ are the orders at the marked points, and $\varepsilon = \pm 1$ is the holonomy.

\begin{theorem}
\label{theorem:main1}
Let $\alpha = (\alpha_1, \dots, \alpha_{k} ; \varepsilon)$ and $(2k_0 - 2g + \varepsilon + 1) > 0$. Then a set of simple closed curves $\Sigma \subset \mathcal{S}(S)$ is spectrally rigid over $\mathsf{Flat}(S,\alpha)$ if and only if $\overline{\Sigma} = \mathcal{PMF}(S)$.
\end{theorem}

The key technical result needed for this is a new, more flexible construction of deformation families of constant $\Sigma$-length-spectrum. In \cite{DLR}, the authors used a constructive case-by-case proof to show the existence of such deformation families. For closed surfaces of genus $g$, they constructed deformation families of dimension $(2g-3)$. Ours is an existence proof which allows for a unified treatment for all surfaces, and all applicable strata. One direction of Theorem~\ref{theorem:main1} follows from \cite{DLR}*{Theorem~2}. The other direction comes as the corollary of the following theorem.

\begin{theorem}
\label{theorem:main11}
Let $\alpha = (\alpha_1,\dots,\alpha_{k} ; \varepsilon)$ and $(2k_0 - 2g + \varepsilon + 1) > 0$, and suppose $\Sigma$ is a set of simple closed curves with $\overline{\Sigma} \neq \mathcal{PMF}(S)$. Then there exists a deformation family $\Omega_\Sigma \subset \mathsf{Flat}(S,\alpha)$ such that the length function $\lambda_\Sigma$ is constant on $\Omega_\Sigma$ and $\dim(\Omega_\Sigma) \geq (2k_0 - 2g + \varepsilon + 1)$. Consequently, there exists a deformation family $\Omega_\Sigma \subset \mathsf{Flat}(S)$ of dimension at least $(6g + 2n - 8)$.
\end{theorem}

The results above show that a finite set of closed curves can never be spectrally rigid when restricted to a single stratum with sufficiently high dimension. This is counterintuitive since these spaces of metrics are all finite dimensional. To complement Theorem~\ref{theorem:main1}, we prove that a finite set of closed curves can be \emph{locally} spectrally rigid over a stratum.

\begin{theorem}
\label{theorem:main2}
Let $\alpha = (\alpha_1, \dots, \alpha_{k} ; \varepsilon)$. For any $\rho \in \mathsf{Flat}(S,\alpha)$, there exists a set of closed curves $\Sigma \subset \mathcal{C}(S)$ such that $\Sigma$ is locally spectrally rigid at $\rho \in \mathsf{Flat}(S,\alpha)$ and $|\Sigma| \leq 15(2g + k - 2)$.
\end{theorem}

\noindent \textbf{Outline of paper.} In Section~\ref{qdfm} we define quadratic differentials and the associated flat metrics. In Section~\ref{qdhc} we recall the construction of holonomy coordinates for a general stratum of quadratic differentials. Using holonomy coordinates, we prove that the length function of a closed curve is smooth on a dense open set; see Proposition~\ref{stable}. Section~\ref{mftt} describes the relevant facts regarding measured foliations and train tracks. Theorem~\ref{theorem:main11} is proved in Section~\ref{thm1}. The key ingredient of the proof is the construction of a function $f$ mapping an open set of $QD(S,\alpha)$ to the weight space $W_\tau$ of a maximal train track. The length of any curve $\gamma \in \Sigma$ with respect to a flat metric induced by $q\in QD(S,\alpha)$ is equal to the intersection number $i(f(q),\gamma)$ by construction. The fibers of $f$ projects to deformation families of flat metrics of constant $\Sigma$-length-spectrum. Theorem~\ref{theorem:main2} is proved in Section~\ref{thm2}, where we observe that flat metrics in a stratum are locally determined by the lengths of saddle connections. We then construct, for any saddle connection, a set of at most 5 closed curves whose lengths determine the length of the saddle connection on a sufficiently small neighborhood of the given flat metric.

\noindent \textbf{Acknowledgment.} The author would like to thank his advisor Chris Leininger for all the guidance in the process of this work. The author would also like to thank the anonymous referee for the helpful comments.


\section{Quadratic Differentials and Flat Metrics}\label{qdfm}

Let $S$ be a closed surface of genus $g$ with $n$ marked points, and $\hat{S}$ the surface obtained by removing the marked points. By a \emph{quadratic differential} on $S$ we mean a complex structure on $S$ together with an integrable nonzero meromorphic quadratic differential. The quadratic differential is allowed to have poles of order at most one at marked points and is assumed to be holomorphic on $\hat{S}$. The space of all quadratic differentials, defined up to isotopy rel marked points, is denoted $QD(S)$. A point of $QD(S)$ will be denoted $q$, with the underlying complex structure implicit in the notation. Let $P$ be the set of marked points and $C_0(q)$ be the set of zeroes of $q$ at non-marked points. The set $C(q)=C_0(q)\cup P$ is called the set of \emph{cone points} of $q$. We use $C_0$ and $C$, suppressing the dependence on the quadratic differential when it is safe to do so. For more on quadratic differentials, the reader is referred to \cite{S}.

Integrating a square root of a nonzero quadratic differential $q$ in a small neighborhood of a point where $q$ is nonzero produces natural coordinates $z$ on $\hat{S}$ in which $q=dz^2$. The collection of all natural coordinates determines a \emph{semi-translation structure} on $S \smallsetminus C(q)$. This is an open cover $\{ U_{\beta} \}$ of $S \smallsetminus C(q)$ along with charts $\phi_\beta : U_\beta \to \mathbb{R}^2$ so that for every $\beta$, $\delta$ with $U_\beta \cap U_\delta \neq \emptyset$, we have
\[
\phi_\beta \circ \phi_\delta^{-1}(v) = \pm v + c, \text{ where }v,c \in \mathbb{R}^2.
\]
This induces a Euclidean metric on $S \smallsetminus C(q)$. For a semi-translation structure coming from a quadratic differential $q$, the metric extends to a Euclidean cone metric on all of $S$ so that at a zero of order $d$, one has a cone point with cone angle $(2+d)\pi$. If the point is a pole (and hence a marked point), we view this as a zero of order $-1$, and then the cone angle is $\pi$.

Denote the set of isotopy classes of unit area Euclidean cone metrics induced by quadratic differentials as $\mathsf{Flat}(S)$. There is a map $QD(S)\to \mathsf{Flat}(S)$ which is the quotient by rotation and scales to have area 1. Any quadratic differential $q\in QD(S)$ can be obtained as Euclidean polygons with isometric side-gluings. The flat metric $\rho \in \mathsf{Flat}(S)$ induced by $q$ is obtained by scaling the Euclidean polygons to have total area 1. 

The space of quadratic differentials $QD(S)$ is naturally stratified by the number of zeroes, the order of the zeroes, and the holonomy. Let $\alpha = (\alpha_1, \alpha_2, \dots, \alpha_{k}; \varepsilon)$ where $k=k_0+n$, 
\[
\alpha_1 \geq \dots \geq \alpha_{k_0} \geq 1 \text{, } \alpha_{k_0+1} \geq \dots \geq \alpha_{k} \geq -1 \text{, } \sum_{i=1}^{k} \alpha_i = 4g-4,
\] 
and $\varepsilon = \pm 1$ denotes the holonomy of the metric in $S \smallsetminus C(q)$. Each $\alpha_i$ with $1 \leq i \leq k_0$ corresponds to a zero of order $\alpha_i$ and each $\alpha_i$ with $k_0+1 \leq i \leq k$ is the order at a marked point. We write $QD(S,\alpha)$ to denote the stratum of $QD(S)$ consisting of quadratic differentials whose zeroes/poles and holonomy are given by $\alpha$. This descends to a stratification of $\mathsf{Flat}(S)$. Each $\alpha_i$ corresponds to a cone point with cone angle $(\alpha_i+2)\pi$. The holonomy is described by $\varepsilon$ as above. We denote the stratum of $\mathsf{Flat}(S)$ defined by $\alpha$ as $\mathsf{Flat}(S,\alpha)$.

Let $\mathcal{C}(S)$ be the set of homotopy classes of closed curves on $\hat{S}$ and $\mathcal{S}(S)$ be the set of homotopy classes of simple closed curves on $\hat{S}$, where all the curves are assumed to be essential (non-nullhomotopic and non-peripheral). For $c \in \mathcal{C}(S)$, the length of $c$ with respect to $\rho \in \mathsf{Flat}(S)$ is the infimum of lengths of representatives of $c$ in a representative of $\rho$, denoted by $\ell(c,\rho)$. In fact, there exists a closed geodesic with respect to $\rho$ that realizes the length $\ell(c,\rho)$ though it may lie in $S$ rather than $\hat{S}$. This is obtained as the limit of representatives of $c$ in $\hat{S}$ with lengths approaching $\ell(c,\rho)$.

We digress a bit here to talk about what geodesics look like with respect to metrics in $\mathsf{Flat}(S)$. Let $c \in \mathcal{C}(S)$ be a closed curve on $\hat{S}$. For a fixed $\rho \in \mathsf{Flat}(S)$ induced by $q \in QD(S)$, $c$ is a geodesic with respect to $\rho$ if and only if $c$ is a concatenation of straight line segments connecting cone points in $S \smallsetminus C(q)$, called \emph{saddle connections}, and $c$ satisfies the \emph{angle condition} at each point in $C_0$. The angle condition requires that each time $c$ passes through a point in $C_0$, the angles on both sides of $c$ must be greater or equal to $\pi$. If $c$ is a closed curve on $S$ obtained as a limit of representatives of a homotopy class in $\hat{S}$, $c$ is a geodesic with respect to $\rho$ if and only if $c \smallsetminus P$ satisfies the conditions above and at points in $P$ the angle condition is satisfied on the ``side of $c$ opposite the point in $P$''; see \cite{S}*{Section~8} for a more precise description. See Figure~\ref{geod} for examples.

\begin{figure}[h]
\begin{tikzpicture}[scale=1.5]
\draw (-2,-.8) node{Away from zeroes};
 \draw [line width=.5mm] [->] (-3,0.5) -- (-1,1.5); 
 \draw [blue] (-3,0.25) .. controls (-2,0) and (-2,2) .. (-1,1.75);
\draw (1,-.8) node{Cone angle $4\pi$};
 \draw [line width=.5mm] [->] (-0.25,0.375) -- (.5,.75); 
 \draw [line width=.5mm] (.5,0.75) -- (1,1); 
 \draw [line width=.5mm] [->] (1,1) -- (2.25,.5); 
 \draw [blue] (-0.25,0.175) .. controls (1.25,2.15) and (1.75,.65) .. (2.25,.25);
 \draw (-.25,-.25) -- (2.25,2.25);
 \draw (-.25,2.25) -- (2.25,-.25);
 \draw (.75,.875) arc (195:-15:.3);
 \draw (1.1,1.35) node{$>\pi$};
 \draw (.7,.85) arc (195:348:.3);
 \draw (1,.5) node{$>\pi$};
\draw (4,-.8) node{Marked point};
 \draw [line width=.5mm] [->] (2.75,1.0625) -- (3.5,1.025); 
 \draw [line width=.5mm] (3.5,1.025) -- (3.96,1.02); 
 \draw [line width=.5mm] [->] (4.04,1.03) -- (5,2); 
 \filldraw [gray] (4,1) circle (.05);
 \draw (3.7,1.02) arc (178:450:0.28);
 \draw (4,1.29) arc (90:270:0.23);
 \draw (4,0.83) arc (270:398:0.19);
 \draw (4,.65) node{$>\pi$};
 \draw [blue] (2.75,1.2625) .. controls (4.65,-.75) and (5,1.5) .. (4,1.5);
 \draw [blue] (4,1.5) .. controls (2.5,1.25) and (4.35,-.8) .. (5,1.8);
\end{tikzpicture}
\caption{Examples of curves and geodesic representatives.}
\label{geod}
\end{figure}
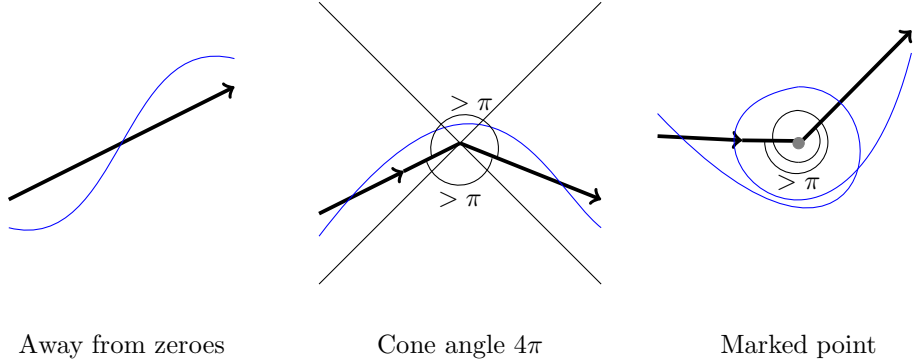

The \emph{length function} of $\Sigma \subset \mathcal{C}(S)$ with respect to $\mathsf{Flat}(S,\alpha)$ is the function
\[
\lambda_\Sigma : \mathsf{Flat}(S,\alpha) \to \mathbb{R}^\Sigma \text{, with } \lambda_\Sigma(\rho) = \big( \ell(c,\rho) \big)_{c \in \Sigma}.
\] 
The \emph{$\Sigma$-length-spectrum} of $\rho \in \mathsf{Flat}(S,\alpha)$ is $\lambda_\Sigma(\rho)$, an element of $\mathbb{R}^\Sigma$. Our goal is to investigate the question: when is the length function of some set $\Sigma \subset \mathcal{C}(S)$ injective? If not, when is the length function locally injective? We rephrase the question as a \emph{length spectral rigidity} problem. The set $\Sigma \subset \mathcal{C}(S)$ is called \emph{spectrally rigid} over $\mathsf{Flat}(S,\alpha)$ if the length function $\lambda_\Sigma$ is injective. We say $\Sigma$ is \emph{locally spectrally rigid} at $\rho \in \mathsf{Flat}(S,\alpha)$ if there exists a neighborhood $U_\rho \subset \mathsf{Flat}(S,\alpha)$ of $\rho$ such that the length function $\lambda_\Sigma\big|_{U_\rho}$ is injective. The question becomes one of characterizing (locally) spectrally rigid sets of closed curves. The length spectral rigidity problem can be asked for more general families of metrics. For classical results and more motivation see \cite{DLR}*{Introduction}.


\section{Quadratic Differentials and Holonomy Coordinates}\label{qdhc}

We introduce Abelian differentials which are closely related to quadratic differentials. By an \emph{Abelian differential} on $S$ we mean a complex structure on $S$ together with a nonzero holomorphic 1-form. From an Abelian differential on $S$ we can obtain a collection of charts on the complement of the zeroes for which the transition maps are translations. Therefore an Abelian differential $\omega$ on $S$ with zeroes $C_0(\omega)$ determines a \emph{translation structure} on $S$, which is an open cover $\{ U_{\beta} \}$ of $S \smallsetminus C(\omega)$ along with charts $\phi_\beta : U_\beta \to \mathbb{R}^2$ so that for every $\beta$, $\delta$ with $U_\beta \cap U_\delta \neq \emptyset$, we have
\[
\phi_\beta \circ \phi_\delta^{-1}(v) = v + c, \text{ where }v,c \in \mathbb{R}^2.
\]
Let $\mathcal{H}(S)$ denote the space of Abelian differentials on $S$ up to isotopy rel marked points. The space $\mathcal{H}(S)$ is stratified by the number and the order of the zeroes. We use $\mathcal{H}(S,\alpha)$ to denote the stratum of $\mathcal{H}(S)$ defined by $\alpha$. Note that $\alpha$ here includes the number and order of zeroes of the Abelian differential but not the holonomy, which is always trivial.

Given an Abelian differential $\omega \in \mathcal{H}(S,\alpha)$, integrating $\omega$ over relative chains determines a relative cohomology class in $H^1(S,C(\omega);\mathbb{C})$.  This defines a homeomorphism from a neighborhood $U_\omega$ of $\omega$ in $\mathcal{H}(S,\alpha)$ to an open set $V_\omega$ in $H^1(S,C(\omega);\mathbb{C})$.  Next, given a basis of the relative homology $H_1(S,C(\omega);\mathbb{C})$, this determines an isomorphism $H^1(S,C(\omega);\mathbb{C}) \to \mathbb{C}^m$ for some $m$.  Composing with the homeomorphism $U_\omega \to V_\omega$, we obtain \emph{holonomy coordinates} $U_\omega \to \mathbb{C}^m$; see \cite{Z}*{Section~3.3} for more details.

Now let $QD(S,\alpha)$ be a stratum with trivial holonomy, that is, $\varepsilon = 1$.  For any $q \in QD(S,\alpha)$, there exists $\omega \in \mathcal{H}(S,\alpha')$ such that $q = \omega^2$.  This $\omega$ is unique up to sign, and we can find neighborhoods $U_q \subset QD(S,\alpha)$ of $q$ and $U_\omega \subset \mathcal{H}(S,\alpha')$ of $\omega$ so that $U_\omega \to U_q$ defined by squaring Abelian differentials defines a homeomorphism.  If $U_\omega$ is sufficiently small so that there are holonomy coordinates $U_\omega \to \mathbb{C}^m$, then inverting the homeomorphism $U_\omega \to U_q$ and composing with these coordinates defines \emph{holonomy coordinates} $U_q \to \mathbb{C}^m$ about $q$ in $QD(S,\alpha)$.

In the case when the holonomy is nontrivial, that is, $\varepsilon = -1$, we consider the double (branched) cover of $S$ in which the holonomy becomes trivial. Using a similar idea in \cite{M}, there is a canonical \emph{double cover} $DS \to S$ determined by the holonomy and the semi-translation structure; see Figure~\ref{polygon} for an example. 

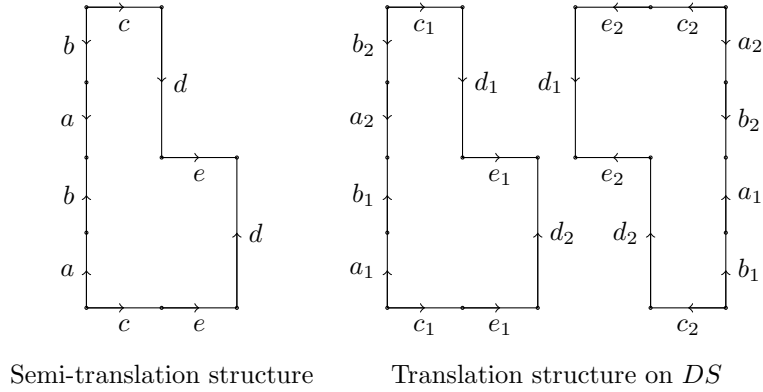
\begin{figure}[ht]
\begin{tikzpicture}[scale=0.5]
\draw (-6,-1.75) node{Semi-translation structure};

\draw (-8,0) -- (-8,8) -- (-6,8) -- (-6,4) -- (-4,4) -- (-4,0) -- (-8,0);
\draw [->] (-8,0) -- (-8,1) node[left=1pt]{$a$};
\draw [->] (-8,2) -- (-8,3) node[left=1pt]{$b$};
\draw [->] (-8,6) -- (-8,5) node[left=1pt]{$a$};
\draw [->] (-8,8) -- (-8,7) node[left=1pt]{$b$};
\draw [->] (-8,8) -- (-7,8) node[below=1pt]{$c$};
\draw [->] (-6,8) -- (-6,6) node[right=1pt]{$d$};
\draw [->] (-6,4) -- (-5,4) node[below=1pt]{$e$};
\draw [->] (-4,0) -- (-4,2) node[right=1pt]{$d$};
\draw [->] (-6,0) -- (-5,0) node[below=1pt]{$e$};
\draw [->] (-8,0) -- (-7,0) node[below=1pt]{$c$};
\draw (-8,0) circle (.04);
\draw (-8,2) circle (.04);
\draw (-8,4) circle (.04);
\draw (-8,6) circle (.04);
\draw (-8,8) circle (.04);
\draw (-6,8) circle (.04);
\draw (-6,4) circle (.04);
\draw (-4,4) circle (.04);
\draw (-4,0) circle (.04);
\draw (-6,0) circle (.04);

\draw (4.5,-1.75) node{Translation structure on $DS$};

\draw (0,0) -- (0,8) -- (2,8) -- (2,4) -- (4,4) -- (4,0) -- (0,0);
\draw [->] (0,0) -- (0,1) node[left=1pt]{$a_1$};
\draw [->] (0,2) -- (0,3) node[left=1pt]{$b_1$};
\draw [->] (0,6) -- (0,5) node[left=1pt]{$a_2$};
\draw [->] (0,8) -- (0,7) node[left=1pt]{$b_2$};
\draw [->] (0,8) -- (1,8) node[below=1pt]{$c_1$};
\draw [->] (2,8) -- (2,6) node[right=1pt]{$d_1$};
\draw [->] (2,4) -- (3,4) node[below=1pt]{$e_1$};
\draw [->] (4,0) -- (4,2) node[right=1pt]{$d_2$};
\draw [->] (2,0) -- (3,0) node[below=1pt]{$e_1$};
\draw [->] (0,0) -- (1,0) node[below=1pt]{$c_1$};
\draw (0,0) circle (.04);
\draw (0,2) circle (.04);
\draw (0,4) circle (.04);
\draw (0,6) circle (.04);
\draw (0,8) circle (.04);
\draw (2,8) circle (.04);
\draw (2,4) circle (.04);
\draw (4,4) circle (.04);
\draw (4,0) circle (.04);
\draw (2,0) circle (.04);

\draw (5,8) -- (9,8) -- (9,0) -- (7,0) -- (7,4) -- (5,4) -- (5,8);
\draw [->] (9,8) -- (9,7) node[right=1pt]{$a_2$};
\draw [->] (9,6) -- (9,5) node[right=1pt]{$b_2$};
\draw [->] (9,2) -- (9,3) node[right=1pt]{$a_1$};
\draw [->] (9,0) -- (9,1) node[right=1pt]{$b_1$};
\draw [->] (9,0) -- (8,0) node[below=1pt]{$c_2$};
\draw [->] (7,0) -- (7,2) node[left=1pt]{$d_2$};
\draw [->] (7,4) -- (6,4) node[below=1pt]{$e_2$};
\draw [->] (5,8) -- (5,6) node[left=1pt]{$d_1$};
\draw [->] (7,8) -- (6,8) node[below=1pt]{$e_2$};
\draw [->] (9,8) -- (8,8) node[below=1pt]{$c_2$};
\draw (5,8) circle (.04);
\draw (7,8) circle (.04);
\draw (9,8) circle (.04);
\draw (9,6) circle (.04);
\draw (9,4) circle (.04);
\draw (9,2) circle (.04);
\draw (9,0) circle (.04);
\draw (7,0) circle (.04);
\draw (7,4) circle (.04);
\draw (5,4) circle (.04);

\end{tikzpicture}
\caption{Construction of the double cover $DS$.}
\label{polygon}
\end{figure}

We consider $\zeta \in QD(DS)$ that is determined by $q \in QD(S,\alpha)$. Write $\zeta \in QD(DS,\alpha_{DS})$, where $\alpha_{DS}$ describes the stratum $\zeta$ lies in. There exists holonomy coordinates $U_\zeta \to \mathbb{C}^m$ about $\zeta$ as above since $QD(DS,\alpha_{DS})$ is a stratum with trivial holonomy. 

We have two natural maps associated to this setup. The covering map $DS \to S$ and the involution $DS \to DS$ that generates the group of deck transformations. The involution induces an involution on $H^1(DS,C(\zeta);\mathbb{C})$ with eigenvalues $\pm 1$. If $\omega \in \mathcal H(DS,\alpha'_{DS})$ is such that $\omega^2 = \zeta$, then we have a homeomorphism $U_\zeta \to U_\omega \to V_\omega \subset H^1(DS,C(\zeta);\mathbb{C})$ as described above.  This maps the subset of $U_\zeta$ consisting of quadratic differentials in $U_\zeta$ lifted from $S$ precisely onto the intersection with the $(-1)$-eigenspace.  Let $W_q \to U_q \subset U_\zeta$ be a homeomorphism from a neighborhood $W_q \subset QD(S,\alpha)$ to $U_q$.

Given a basis for the relative homology determining holonomy coordinates $U_\zeta \to \mathbb{C}^m$, we can choose a subset of the basis elements so that the map obtained by pairing with only these basis vectors $H^1(DS,C(\zeta);\mathbb{C}) \to \mathbb C^{m'}$ restricts to an isomorphism on the $(-1)$-eigenspace.  The composition $U_\zeta \to V_\omega \to \mathbb C^{m'}$ is therefore a homeomorphism on the subset $U_q \subset U_\zeta$ consisting of quadratic differentials lifted from $S$.   Composing $W_q \to U_q$ with this map, we obtain \emph{holonomy coordinates} $W_q \to U_q \to \mathbb{C}^{m'}$.

We use holonomy coordinates to define local coordinates on $\mathsf{Flat}(S,\alpha)$ as follows. Let $\rho\in\mathsf{Flat}(S,\alpha)$ be induced by $q \in QD(S,\alpha)$. If we consider the open neighborhood $U_q\subset QD(S,\alpha)$ with holonomy coordinates $\varphi:U_q\to\mathbb{C}^m$, then there exists a real codimension two smooth submanifold in $U_q$ that is mapped homeomorphically onto an open set in $\mathsf{Flat}(S,\alpha)$. The submanifold is defined by $\pi_1\circ\varphi(q)$ being a positive real number (which effectively mods out by rotation) and the area of the quadratic differential being 1, where $\pi_1$ is the projection to the first coordinate. Therefore we can obtain from this \emph{holonomy coordinates} about $\rho \in \mathsf{Flat}(S,\alpha)$.

We will need the following result from \cite{Ve}. We provide a proof for completeness. We write $\operatorname{dim}$ for real dimension and $\operatorname{dim}_\mathbb{C}$ for complex dimension.

\begin{prop}\label{dimcount}
Let $\alpha=(\alpha_1,\dots,\alpha_{k};\varepsilon)$. Then 
\[
\operatorname{dim}(QD(S,\alpha)) = 4g+2k+\varepsilon-3 \text{ and }\operatorname{dim}(\mathsf{Flat}(S,\alpha)) = 4g+2k+\varepsilon-5.
\] 
\end{prop}

\begin{proof}
If $\varepsilon=1$, then
\[
\operatorname{dim}_\mathbb{C} H^1(S,C(q);\mathbb{C}) = \operatorname{dim}_\mathbb{C} H_1(S,C(q);\mathbb{C}) = 2g+k-1
\]
implies that
\[
\operatorname{dim}(QD(S,\alpha)) = 2(\operatorname{dim}_{\mathbb{C}}(QD(S,\alpha))) = 4g+2k-2.
\]

For the case when $\varepsilon=-1$, fix $q\in QD(S,\alpha)$ and $\zeta\in QD(DS,\alpha_{DS})$ induced by $q$. Let $k_1$ be the number of $i$ for which $\alpha_i$ is odd and $k_2$ be the number of $i$ for which $\alpha_i$ is even. Each odd $\alpha_i$ corresponds to a $2(\alpha_i+1)$ in $\alpha_{DS}$ and each even $\alpha_i$ corresponds to a pair of $\alpha_i$ in $\alpha_{DS}$. Hence the size of the set of cone points for $DS$ is 
\[
k_{DS}=k_1 + 2k_2,
\]
and the genus of $DS$ is
\[
g_{DS}=2g+\frac{k_1}{2}-1,
\]
obtained by calculating the Euler characteristic using the Riemann-Hurwitz formula. Holonomy coordinates about $\zeta$ has complex dimension $(2g_{DS}+k_{DS}-1)$ as above. The dimension of the $(-1)$-eigenspace is computed by subtracting the dimension of the $(+1)$-eigenspace. The pullback of a basis of $H^1(S,C(q);\mathbb{C})$ to $H^1(DS,C(\zeta);\mathbb{C})$ is invariant under the involution and provides a basis of the $(+1)$-eigenspace. The complex dimension of the $(+1)$-eigenspace is equal to the complex dimension of $H^1(S,C(q);\mathbb{C})$. Therefore
\[
\operatorname{dim}(QD(S,\alpha)) = 2[(2g_{DS}+k_{DS}-1)-(2g+k-1)] = 4g+2k-4.
\]

The space $\mathsf{Flat}(S,\alpha)$ can be obtained by scaling by nonzero complex numbers. Hence $\operatorname{dim}(\mathsf{Flat}(S,\alpha)) = \operatorname{dim}(QD(S,\alpha)) - 2$.
\end{proof}

With holonomy coordinates, we have a better description of the topology of the strata $QD(S,\alpha)$ and $\mathsf{Flat}(S,\alpha)$. Intuitively, an open neighborhood of $q\in QD(S,\alpha)$ can be imagined as perturbing Euclidean polygons with side-gluings. We say a closed geodesic is \emph{stable} inside an open neighborhood if the homotopy classes, rel endpoints, of the saddle connections making up the geodesic is constant. As a consequence, the combinatorics of the saddle connections making up the closed geodesic will be the same inside the open neighborhood.

\begin{prop}\label{stable}
For any closed curve $c$, there exists an open dense subset of $\mathsf{Flat}(S,\alpha)$ in which the geodesic representative of $c$ is stable.
\end{prop}

\begin{proof}
Fix an arbitrary open set $U \subset \mathsf{Flat}(S,\alpha)$. Let $U_0$ be an open neighborhood of an arbitrary $\rho\in U$ with holonomy coordinates $U_0 \to V \subset \mathbb{C}^m$. In holonomy coordinates, the length function of $c$ becomes a function $V \to \mathbb{R}$. Using the description of geodesics in Section~\ref{qdfm}, the length of the closed curve $c$ is a finite sum of Euclidean lengths of saddle connections that belong to a geodesic representative of $c$. 

Let $U_1\subset U_0 \cap U$ be an open set such that the length function of $c$ is bounded by some $r>0$ on $U_1$. Inside $U_1$ there is a uniform bound $N$ on the number of relative homotopy classes of paths between cone points with length at most $r$ and we denote the set by $\mathcal{A}=\{a_1, \ldots, a_N\}$. Let $U(a_0) = U_1$ and define the following sets iteratively by 
\[
U(a_{i}) = \{ \rho' \in U(a_{i-1}) \mid a_{i} \text{ is represented by a saddle connection in } \rho'\}
\]
or $U(a_{i}) = U(a_{i-1})$ if the above set is empty. The terminal set $U(a_N)$ is open and nonempty. Call the subset of $\mathcal{A}$ that is represented by a saddle connection $X$. By construction, all saddle connections $X$ remain saddle connections in $U(a_N)$

For every two saddle connections $s_i$, $s_j$ in $X$, consider the set $NP_{i,j}$ which is the subset of $U(a_N)$ on which $s_i$ and $s_j$ are not parallel, and $P_{i,j}$ the subset on which they are parallel.  Using holonomy coordinates, we see that $P_{i,j}$ is a closed set defined by a single equation, hence $NP_{i,j}$ is open and either dense or empty. Define $E_{i,j}$ to be $NP_{i,j}$ if it is nonempty, and $P_{i,j}$ otherwise (in which case $P_{i,j}$ is all of $U(a_N)$). Then because $X$ is finite, it follows that the intersection of all $E_{i,j}$ is an open dense subset of $U(a_N)$, and we denote this $V$.

Now pick any metric $\rho'$ in $V$. The geodesic representative of $c$ in $\rho'$ is a concatenation of saddle connections of length at most $r$, which therefore come from $X$.  Any two consecutive saddle connections are either nonparallel, and remain so in $V$, or are parallel and remain so in $V$.  It follows that the geodesic representative of $c$ is a concatenation of the same set of saddle connections for every metric $\rho''$ in $V$. Therefore $c$ is a stable closed curve in an open dense subset of $\mathsf{Flat}(S,\alpha)$.
\end{proof}

From Proposition~\ref{stable}, we obtain the following corollary immediately.

\begin{cor}\label{lengsmoo}
Let $U$ be an arbitrary open set in $\mathsf{Flat}(S,\alpha)$.  For any closed curve $c\in\mathcal{C}(S)$ there exists an open nonempty subset $V\subset U$ such that $\ell(c,\cdot)$ is a smooth function over $V$.
\end{cor}


\section{Measured Foliations and Train Tracks}\label{mftt}

In this section we briefly describe measured foliations and train tracks on a surface $S$ of finite type. For detailed and complete descriptions, see \cites{FLP,PH,T}.

A \emph{measured foliation} on $S$ is a singular foliation on $S$ together with a transverse measure $\mu$ of full support without atoms that is invariant under holonomy. The singularities all have negative index except possibly marked points which are allowed to have index 1/2 (see \cite{FLP}). There exist charts to $\mathbb{R}^2$ away from singularities so that horizontal lines describe the foliation, and the transverse measure is given by $|dy|$. We use $\mathcal{MF}(S)$ to denote the set of equivalence classes of measured foliations and $\mathcal{PMF}(S)$ for the set of projective classes of measured foliations.

A quadratic differential on $S$ induces a \emph{vertical measured foliation} on $S$. The vertical measured foliation is obtained by foliating $\mathbb{R}^2$ with vertical lines and we use the semi-translation structure to obtain a measured foliation on the surface $S$. 

Let $\varphi: QD(S) \to \mathcal{MF}(S)$ denote the map that sends a quadratic differential $q$ to the vertical measured foliation. We can associate a circle of measured foliations to $\rho \in \mathsf{Flat}(S)$ induced by $q \in QD(S)$ by considering the circle of all quadratic differentials that induce $\rho$. Define the function 
\[
\nu_{q} : [0,\pi) \to \mathcal{MF}(S)
\] 
where 
\[
\nu_{q}(\theta) = \varphi(e^{2i\theta}q).
\]
A key ingredient of our proof of the main theorem is \cite{DLR}*{Lemma~9}. The lemma provides a formula to compute the length of a closed curve with respect to $\rho \in \mathsf{Flat}(S)$ using the circle of measured foliations associated to $\rho$. We state the lemma here for later reference.

\begin{lemma}\label{ellq}
For all $\rho \in \mathsf{Flat}(S)$ induced by $q \in QD(S)$ and $c \in \mathcal{C}(S)$, we have
\[
\ell(c,\rho) = \frac{1}{2}\int_0^\pi i(\nu_q(\theta),c) d\theta.
\]
\end{lemma}

One way to provide a local description of $\mathcal{MF}(S)$ is via train tracks. In this paper, a train track $\tau$ on $S$ is an embedded trivalent graph with the following extra structure. The edges (called branches) of $\tau$ are smoothly embedded and at each vertex (called a switch) of $\tau$, all adjacent branches have the same tangent line $L$. Furthermore, we require that there is at least one branch on each side of the switch and no component of $S\smallsetminus\tau$ is a nullgon, an unmarked monogon, or an unmarked bigon.

We say that a simple closed curve $c$ is \emph{carried} by a train track $\tau$ if there exists a differentiable map $f:S\to S$ homotopic to the identity rel marked points and $f\mid_c$ is an immersion to $\tau$. The definition extends to measured foliations. A train track $\tau$ \emph{carries} a measured foliation $\mu\in\mathcal{MF}(S)$ if there is a differentiable map $f:S\smallsetminus C_\mu \to S$, where $C_\mu$ is the set of singularities of $\mu$, homotopic to the identity while immersing every leaf of $\mu$ to $\tau$. Let $\mathcal{MF}_\tau(S)$ denote the set of measured foliations on $S$ carried by $\tau$.

A \emph{weight function} on a train track $\tau$ is an assignment of a nonnegative real number to each branch in such a way that the numbers satisfy the \emph{switch condition} at each switch: the sum of the weights on incoming branches equals the sum on outgoing branches. Let $W_\tau$ denote the set of weight functions on $\tau$. We call a function $w\in W_\tau$ a weight function. A simple closed curve $c$ or a measured foliation $\mu$ carried by a train track $\tau$ determines a weight function $w_c$ or $w_\mu$, respectively, on $\tau$.

Two simple closed curves $c_1$ and $c_2$ meet \emph{efficiently} if they meet transversely and there are no unmarked bigon components in $S \smallsetminus (c_1 \cup c_2)$.  In this case, $|c_1 \cap c_2|$ is minimal among all simple closed curves $c_1',c_2'$ homotopic to $c_1,c_2$, respectively. This number is called the \emph{geometric intersection number} and is denoted $i(c_1,c_2)$.

One can similarly define efficient intersection for a curve and a train track, or two train tracks; see \cite{PH}.  When a curve $c$ and train track $\tau$ intersect efficiently, the \emph{intersection number} between $c$ and a weight function $w \in W_\tau$, denoted $i(c,w)$, is the sum of weights of branches over all intersection points of $c$ and $\tau$.  If $c$ and $\tau$ meet efficiently, and $w_{c'}$ is the weight coming from a simple closed curve $c'$, then $i(c,w_{c'}) = i(c,c')$.  Similarly, the intersection number between a measured foliation $\mu$ and a curve $c$ can be defined, and if $\mu$ is carried by $\tau$, then $i(c,\mu) = i(c,w_{\mu})$.  If $\tau$ and $\tau'$ meet efficiently and $w \in W_\tau, w' \in W_{\tau'}$, then $i(w,w')$ is defined as a weighted sum over all intersection points of branches.  Moreover, if these weight functions correspond to curves or measured foliations, then this is the geometric intersection number of the associated objects.

We will construct a pair of train tracks $\tau$ and $\tau'$ meeting efficiently on a surface $S$ and use them throughout the rest of the paper. To describe these, fix a set of \emph{pants curves}, i.e., a maximal set of pairwise disjoint, pairwise non-homotopic simple closed curves on $S$. Let $\tau$ be a train track that contains the set of pants curves, along with the branches shown in Figure~\ref{track} on each pair of pants.

\begin{figure}[ht]
\begin{tikzpicture}[scale=0.15]

\draw (-18,1) arc (180:360:1);
\draw (-17,0) -- (-1,0);
\draw (0,1) arc (0:-180:1);
\draw (-18,1) -- (-18,5);
\draw (-18,5) arc (180:0:1);
\draw (-16,1) -- (-16,5);
\draw (0,1) -- (0,5);
\draw (0,5) arc (0:180:1);
\draw (-2,1) -- (-2,5);
\draw (-17,6) -- (-16,6);
\draw (-1,6) -- (-2,6);
\draw (-16,6) arc (270:360:4);
\draw (-2,6) arc (270:180:4);
\draw (-12,10) -- (-12,13);
\draw (-6,10) -- (-6,13);
\draw (-11,12) arc (270:90:1);
\draw (-7,12) arc (-90:90:1);
\draw (-11,12) -- (-7,12);
\draw (-11,14) -- (-7,14);
\draw [blue] (-10,11) -- (-10,5);
\draw [blue] (-8,11) -- (-8,5);
\draw [blue] (-10,5) arc (360:270:1);
\draw [blue] (-8,5) arc (180:270:1);
\draw [blue] (-11,4) -- (-15,4);
\draw [blue] (-7,4) -- (-3,4);
\draw [blue] (-3,2) -- (-15,2);
\draw [blue] (-3,2) arc (90:0:1);
\draw [blue] (-3,4) arc (90:0:1);
\draw [blue] (-15,2) arc (270:180:1);
\draw [blue] (-15,4) arc (270:180:1);
\draw [blue] (-8,11) arc (180:90:1);
\draw [blue] (-10,11) arc (180:90:1);

\draw (2,1) arc (180:360:1);
\draw (3,0) -- (19,0);
\draw (20,1) arc (0:-180:1);
\draw (2,1) -- (2,5);
\draw (2,5) arc (180:0:1);
\draw (4,1) -- (4,5);
\draw (20,1) -- (20,5);
\draw (20,5) arc (0:180:1);
\draw (18,1) -- (18,5);
\draw (3,6) -- (6.9,6);
\draw (19,6) -- (15.1,6);
\draw (6.9,6) arc (270:360:4);
\draw (15.1,6) arc (270:180:4);
\draw (11,10) circle (0.1);
\draw [blue] (17,2) -- (5,2);
\draw [blue] (5,2) arc (270:180:1);
\draw [blue] (17,2) arc (90:0:1);
\draw [blue] [dotted] (7,6) -- (14,6);
\draw [blue] (9,4) arc (270:180:2);
\draw [blue] (14,6) arc (180:270:2);
\draw [blue] (9,4) -- (17,4);
\draw [blue] (17,4) arc (90:0:1);

\draw (22,1) arc (180:360:1);
\draw (23,0) -- (36,0);
\draw (22,1) -- (22,5);
\draw (22,5) arc (180:0:1);
\draw (24,1) -- (24,5);
\draw (23,6) -- (36,6);
\draw (36,0.1) circle (0.1);
\draw (36,5.9) circle (0.1);
\draw (36,0.2) arc (270:180:2);
\draw (36,5.8) arc (90:180:2);
\draw (34,2.2) -- (34,3.8);
\draw [blue] (34,2) -- (25,2);
\draw [blue] [dotted] (34,2) arc (360:270:2);
\draw [blue] (32,0) arc (0:90:2);
\draw [blue] (25,2) arc (270:180:1);

\end{tikzpicture}
\caption{The train track $\tau$ on pairs of pants.}
\label{track}
\end{figure}

The train track $\tau$ has the property that every branch that is going toward a pants curve is turning to its right. We similarly construct $\tau'$ using the same pants curves, but using left turns going toward any pants curve; see Figure~\ref{track'}. 

\begin{figure}[ht]
\begin{tikzpicture}[scale=0.15]

\draw (-18,1) arc (180:360:1);
\draw (-17,0) -- (-1,0);
\draw (0,1) arc (0:-180:1);
\draw (-18,1) -- (-18,5);
\draw (-18,5) arc (180:0:1);
\draw (-16,1) -- (-16,5);
\draw (0,1) -- (0,5);
\draw (0,5) arc (0:180:1);
\draw (-2,1) -- (-2,5);
\draw (-17,6) -- (-16,6);
\draw (-1,6) -- (-2,6);
\draw (-16,6) arc (270:360:4);
\draw (-2,6) arc (270:180:4);
\draw (-12,10) -- (-12,13);
\draw (-6,10) -- (-6,13);
\draw (-11,12) arc (270:90:1);
\draw (-7,12) arc (-90:90:1);
\draw (-11,12) -- (-7,12);
\draw (-11,14) -- (-7,14);
\draw [blue] (-10,11) -- (-10,5);
\draw [blue] (-8,11) -- (-8,5);
\draw [blue] (-10,5) arc (360:270:1);
\draw [blue] (-8,5) arc (180:270:1);
\draw [blue] (-11,4) -- (-15,4);
\draw [blue] (-7,4) -- (-3,4);
\draw [blue] (-3,2) -- (-15,2);
\draw [blue] (-3,2) arc (-90:0:1);
\draw [blue] (-3,4) arc (-90:0:1);
\draw [blue] (-15,2) arc (90:180:1);
\draw [blue] (-15,4) arc (90:180:1);
\draw [blue] (-8,11) arc (0:90:1);
\draw [blue] (-10,11) arc (0:90:1);

\draw (2,1) arc (180:360:1);
\draw (3,0) -- (19,0);
\draw (20,1) arc (0:-180:1);
\draw (2,1) -- (2,5);
\draw (2,5) arc (180:0:1);
\draw (4,1) -- (4,5);
\draw (20,1) -- (20,5);
\draw (20,5) arc (0:180:1);
\draw (18,1) -- (18,5);
\draw (3,6) -- (6.9,6);
\draw (19,6) -- (15.1,6);
\draw (6.9,6) arc (270:360:4);
\draw (15.1,6) arc (270:180:4);
\draw (11,10) circle (0.1);
\draw [blue] (17,2) -- (5,2);
\draw [blue] (5,2) arc (90:180:1);
\draw [blue] (17,2) arc (-90:0:1);
\draw [blue] [dotted] (7,6) -- (14,6);
\draw [blue] (9,4) arc (270:180:2);
\draw [blue] (14,6) arc (180:270:2);
\draw [blue] (9,4) -- (17,4);
\draw [blue] (17,4) arc (-90:0:1);

\draw (22,1) arc (180:360:1);
\draw (23,0) -- (36,0);
\draw (22,1) -- (22,5);
\draw (22,5) arc (180:0:1);
\draw (24,1) -- (24,5);
\draw (23,6) -- (36,6);
\draw (36,0.1) circle (0.1);
\draw (36,5.9) circle (0.1);
\draw (36,0.2) arc (270:180:2);
\draw (36,5.8) arc (90:180:2);
\draw (34,2.2) -- (34,3.8);
\draw [blue] (34,2) -- (25,2);
\draw [blue] [dotted] (34,2) arc (360:270:2);
\draw [blue] (32,0) arc (0:90:2);
\draw [blue] (25,2) arc (90:180:1);

\end{tikzpicture}
\caption{The train track $\tau'$ on pairs of pants.}
\label{track'}
\end{figure}

The complement of $\tau$ (and likewise $\tau'$) consists of trigons and/or marked monogons. The pair $\tau$ and $\tau'$ are constructed to be \emph{maximal standard} train tracks as in \cite{PH}. By construction, $\tau$ and $\tau'$ satisfy the following properties.

\begin{enumerate}
\item By applying a suitable isotopy, we can assume that $\tau$ and $\tau'$ meet efficiently.
\item By \cite{PH}*{Section~1.7}, $\mathcal{MF}_\tau(S)$ and $\mathcal{MF}_{\tau'}(S)$ are both nonempty open sets homeomorphic to $W_{\tau}$ and $W_{\tau'}$ respectively. The homeomorphism is given by sending $\mu$ to the weight vector $w_\mu$ defined by the carrying.
\end{enumerate}

From the homeomorphism $W_\tau \to \mathcal{MF}_\tau(S)$, we see that any simple closed curve $c$ naturally gives rise to a measured foliation $\mu_c$ (there is also a direct construction of $\mu_c$ from $c$). This defines a map $\mathcal{S}(S) \to \mathcal{MF}(S)$.  Projectivizing produces an injection $\mathcal{S}(S) \to \mathcal{PMF}(S)$ onto a dense set, by a result of Thurston.

The homeomorphism between $\mathcal{MF}_\tau(S)$ and $W_{\tau}$ is used in the key step of the proof of Theorem~\ref{theorem:main11}. In the next proposition, we prove that there exists a set $\sigma$ of simple closed curves such that the intersection function 
\[
i(\sigma,\cdot) : W_\tau \to \mathbb{R}^\sigma \text{, where } i(\sigma,w) = \left( i(c,w) \right)_{c\in\sigma},
\]
will provide a global coordinate for $W_\tau$. The dimension of $\mathcal{MF}_\tau(S)$ is equal to $(6g+2n-6)$, which will be the same as the size of $\sigma$.

\begin{prop}\label{Wcoord}
Let $\tau$ be the train track constructed above. Then there exists a set of $(6g+2n-6)$ simple closed curves $\sigma$ such that $i(\sigma,\cdot):W_\tau \to \mathbb{R}^{\sigma}_{\geq 0}$ is an injective linear map. Consequently, there exists a linear inverse $A : \operatorname{Image}(i(\sigma,\cdot)) \to W_\tau$.
\end{prop}

\begin{proof}
We will find $\sigma$ so that for any $w \in W_\tau$ we can solve for the weight assigned to each branch by the vector $i(\sigma,w)$. The first $(3g+n-3)$ curves of our set $\sigma$ will be the pants curves, which we denote $\sigma_P$. To determine the weights on the branches interior to the pairs of pants, we divide the analysis into three cases according to Figure~\ref{track}.

Suppose $c_1,c_2,c_3 \in \sigma_P$ are three pants curves bounding a single pair of pants without any marked points. We can see that $i(c_j,w)$ is the sum of the weights on two of the branches interior to the pair of pants. Therefore by using the three pants curves, we obtain the weights on three branches to be of the form 
\[
\frac{1}{2}(i(c_{j_1},w)+i(c_{j_2},w)-i(c_{j_3},w)) \text{ where } \{j_1, j_2, j_3\} = \{1, 2, 3\}.
\]

For the case of a pair of pants with one marked point, we only get two pants curves $c_1$ and $c_2$. Let us assume that in Figure~\ref{track}, $c_1$ is on the left and $c_2$ is on the right. We see that $i(c_1,w)$ is equal to the weight on the branch that connects the two pants curves. The weight on the branch that wrapped around the marked point is half the weight of the branch connecting it to $c_2$, which is
\[
i(c_2,w) - i(c_1,w).
\]

The last case is a pair of pants with two marked points which has only one pants curve $c_1$. The weight on the branch that wrapped around the marked point is half the weight of the branch connecting it to $c_1$, which is just $i(c_1,w)$. Therefore with $(3g+n-3)$ curves we can find the weights of $w$ on the branches interior to each pair of pants.

Now we need to find another set of $(3g+n-3)$ simple closed curves to determine the weights on the branches of $\tau$ that are contained in the pants curves. We take a transverse simple closed curve for each pants curve, disjoint from any other pants curves, meeting $\tau$ efficiently; see Figure~\ref{transverse} for the situation when the pairs of pants contain no marked points.

\begin{figure}[ht]
\begin{tikzpicture}[scale=0.15]

\draw (11,-18) arc (-90:90:1);
\draw (12,-17) -- (12,-1);
\draw (11,0) arc (90:-90:1);
\draw (11,-18) -- (7,-18);
\draw (7,-18) arc (270:90:1);
\draw (11,-16) -- (7,-16);
\draw (11,0) -- (7,0);
\draw (7,0) arc (90:270:1);
\draw (11,-2) -- (7,-2);
\draw (6,-17) -- (6,-16);
\draw (6,-1) -- (6,-2);
\draw (6,-16) arc (0:90:4);
\draw (6,-2) arc (360:270:4);
\draw (2,-12) -- (0,-12);
\draw (2,-6) -- (0,-6);
\draw [line width=.5mm] (0,-12) -- (0,-6);
\draw (0,-5) node{$c'$};

\draw [blue] (0.8,-10) -- (7,-10);
\draw [blue] (1,-8) -- (7,-8);
\draw [blue] (7,-10) arc (90:0:1);
\draw [blue] (7,-8) arc (270:360:1);
\draw [blue] (8,-11) -- (8,-15);
\draw [blue] (8,-7) -- (8,-3);
\draw [blue] (10,-3) -- (10,-15);
\draw [blue] (10,-3) arc (180:90:1);
\draw [blue] (8,-3) arc (180:90:1);
\draw [blue] (10,-15) arc (360:270:1);
\draw [blue] (8,-15) arc (360:270:1);
\draw [blue] (1,-8) arc (270:180:1);
\draw [blue] (0.8,-10) arc (270:180:0.8);

\draw [red,line width=.5mm] (-11,-9) -- (11,-9);
\draw [red,line width=.5mm] (11,-9) arc (270:360:1);
\draw [red,dotted] (12,-8) arc (0:90:1);
\draw [red,dotted] (11,-7) -- (-11,-7);
\draw [red,dotted] (-11,-7) arc (90:180:1);
\draw [red,line width=.5mm] (-12,-8) arc (180:270:1);
\draw (13,-8) node{$c$};

\draw (-11,-18) arc (270:90:1);
\draw (-12,-17) -- (-12,-1);
\draw (-11,0) arc (90:270:1);
\draw (-11,-18) -- (-7,-18);
\draw (-7,-18) arc (-90:90:1);
\draw (-11,-16) -- (-7,-16);
\draw (-11,0) -- (-7,0);
\draw (-7,0) arc (90:-90:1);
\draw (-11,-2) -- (-7,-2);
\draw (-6,-17) -- (-6,-16);
\draw (-6,-1) -- (-6,-2);
\draw (-6,-16) arc (180:90:4);
\draw (-6,-2) arc (180:270:4);
\draw (-2,-12) -- (0,-12);
\draw (-2,-6) -- (0,-6);
\draw (0,-12) -- (0,-6);
\draw [blue] (-1,-10) -- (-7,-10);
\draw [blue] (-0.8,-8) -- (-7,-8);
\draw [blue] (-7,-10) arc (90:180:1);
\draw [blue] (-7,-8) arc (270:180:1);
\draw [blue] (-8,-11) -- (-8,-15);
\draw [blue] (-8,-7) -- (-8,-3);
\draw [blue] (-10,-3) -- (-10,-15);
\draw [blue] (-10,-3) arc (180:90:1);
\draw [blue] (-8,-3) arc (180:90:1);
\draw [blue] (-10,-15) arc (0:-90:1);
\draw [blue] (-8,-15) arc (0:-90:1);
\draw [blue] (-0.8,-8) arc (90:0:0.8);
\draw [blue] (-1,-10) arc (90:0:1);

\draw (17,-7) -- (17,-11);
\draw (19,-7) -- (19,-11);
\draw (17,-7) arc (180:0:1);
\draw (19,-11) arc (360:180:1);
\draw (18,-6) -- (20,-6);
\draw (18,-12) -- (20,-12);
\draw (20,-6) arc (270:360:2);
\draw (20,-12) arc (90:0:2);
\draw (22,-4) arc (180:90:4);
\draw (22,-14) arc (180:270:4);
\draw (26,0) -- (30,0);
\draw (26,-18) -- (30,-18);
\draw (30,0) arc (90:0:4);
\draw (30,-18) arc (270:360:4);
\draw (34,-4) -- (34,-14);
\draw [line width=.5mm] (34,-9) node[right]{$c'$} -- (30,-9);
\draw (30,-7) arc (0:180:2);
\draw (30,-7) -- (30,-11);
\draw (26,-7) -- (26,-11);
\draw (26,-11) arc (180:360:2);
\draw [blue] (31,-8.2) arc (360:270:0.8);
\draw [blue] (31,-8.2) -- (31,-7);
\draw [blue] (31,-7) arc (0:180:3);
\draw [blue] (25,-7) -- (25,-11);
\draw [blue] (25,-11) arc (180:360:3);
\draw [blue] (31,-11) -- (31,-9.8);
\draw [blue] (31,-9.8) arc (180:90:0.8);
\draw [red,line width=.5mm] (32,-7) -- (32,-11);
\draw [red,line width=.5mm] (32,-7) arc (0:180:4);
\draw [red,line width=.5mm] (24,-7) -- (24,-11);
\draw [red,line width=.5mm] (24,-11) arc (180:360:4);
\draw (23,-9) node{$c$};
\draw [blue] (33,-8.2) arc (360:270:0.8);
\draw [blue] (33,-8.2) -- (33,-6);
\draw [blue] (33,-6) arc (0:180:5);
\draw [blue] (23,-12) arc (180:360:5);
\draw [blue] (33,-12) -- (33,-9.8);
\draw [blue] (33,-9.8) arc (180:90:0.8);
\draw [blue] (23,-6) arc (360:270:2);
\draw [blue] (23,-12) arc (0:90:2);
\draw [blue] (21,-8) -- (20,-8);
\draw [blue] (21,-10) -- (20,-10);
\draw [blue] (20,-8) arc (270:180:1);
\draw [blue] (20,-10) arc (270:180:1);

\end{tikzpicture}
\caption{Examples of transverse curves.}
\label{transverse}
\end{figure}
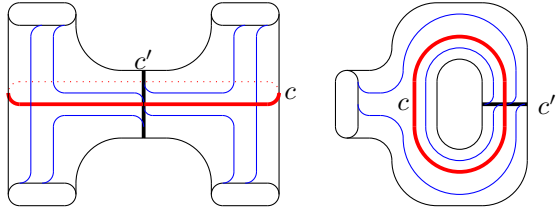

Let $c$ be such a transverse simple closed curve for the pants curve $c'$. We can use $i(c,w)$ to find the weights on branches on the pants curve, which are all expressed by $i(c,w)$ combined with constants determined by $i(\sigma_P,w)$. In Figure~\ref{eqns}, $w_5,\ldots,w_{10}$ had been determined by $\sigma_P$, and the equations given by switch conditions allow us to solve for $w_1,\ldots,w_4$ from these and $i(c,w)$. Explicitly we have
\begin{align*}
w_1 &=\frac{1}{2}(i(c,w)+w_5-w_6-w_9-w_{10}),\\
w_2 &=\frac{1}{2}(i(c,w)-w_5-w_6-w_9-w_{10}),\\
w_3 &=\frac{1}{2}(i(c,w)-w_5+w_6-w_9-w_{10}),\\
w_4 &=\frac{1}{2}(i(c,w)-w_7-w_8-w_9-w_{10}).
\end{align*}
This is true in general. For any transverse simple closed curve $c$ of the pants curve $c'$, let $w'$ be the weight on a branch on $c'$, then
\[
w' = \frac{1}{i(c,c')}\left(i(c,w) + L(i(\sigma_P,w))\right),
\]
where $i(c,c')=1$ or $2$ is the geometric intersection number and $L$ is a linear function from $\mathbb{R}^{3g+n-3}$ to $\mathbb{R}$. This completes the proof of injectivity since $w\in W_\tau$ is uniquely determined by the intersection of $w$ with the set of $(6g+2n-6)$ simple closed curves.

\begin{figure}
\begin{tikzpicture}[scale=0.5]

\draw (3,-12) -- (-3,-12);
\draw (3,-6) -- (-3,-6);
\draw (0,-6) node[above]{$c'$};
\draw [blue] (0.8,-10) -- (2,-10) node[right]{$w_7$};
\draw [blue] (1,-8) -- (2,-8) node[right]{$w_5$};
\draw [blue] (1,-8) arc (270:180:1);
\draw [blue] (0.8,-10) arc (270:180:0.8);
\draw [blue] (-3.8,-7) -- (-3.8,-11) node[below]{$w_{10}$};
\draw [blue] (3.8,-7) -- (3.8,-11) node[below]{$w_9$};
\draw [red] (-4.5,-9.25) -- (4.5,-9) node[right]{$c$};
\draw [red,dotted] (-4.5,-8.75) -- (4.5,-8.5);
\draw (0,-11.5) node[left]{$w_1$} -- (0,-10.3) node[right]{$w_4$} -- (0,-9) node[left]{$w_3$} -- (0,-7.7) node[left]{$w_2$} -- (0,-6.5) node[left]{$w_1$};
\draw (0,-11.5) arc (180:270:0.5);
\draw (0,-6.5) arc (180:90:0.5);
\draw [dotted] (1,-11.5) arc (360:270:0.5);
\draw [dotted] (1,-6.5) arc (0:90:0.5);
\draw [dotted] (1,-6.5) -- (1,-11.5);
\draw [blue] (-1,-10) -- (-2,-10) node[left]{$w_8$};
\draw [blue] (-0.8,-8) -- (-2,-8) node[left]{$w_6$};
\draw [blue] (-0.8,-8) arc (90:0:0.8);
\draw [blue] (-1,-10) arc (90:0:1);

\draw (8,-7) node{$w_1=w_2+w_5$};
\draw (8,-8) node{$w_3=w_2+w_6$};
\draw (8,-9) node{$w_3=w_4+w_7$};
\draw (8,-10) node{$w_1=w_4+w_8$};
\draw (9.3,-11) node{$i(c,w)=w_1+w_3+w_9+w_{10}$};

\end{tikzpicture}
\caption{Explicit weights and equations near a pants curve.}
\label{eqns}
\end{figure}
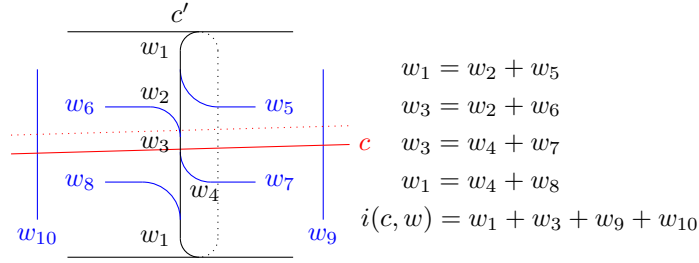

The set of $(3g+n-3)$ transverse curves $\sigma_T$ union with the $(3g+n-3)$ pants curves $\sigma_P$ will be our $\sigma$. We make a further remark that if $W_\tau$ is seen as $\mathbb{R}^{e(\tau)}$ where $e(\tau)$ is the number of branches of $\tau$, then there exists an $e(\tau)$ by $(6g+2n-6)$ matrix $A$ that maps $i(\sigma,w)$ to $w\in W_\tau$.
\end{proof}

We wrap up this section by considering the action of the mapping class group on measured foliations and train tracks. Pseudo-Anosov elements of the mapping class group act with \emph{north-south dynamics} on $\mathcal{PMF}(S)$. That is, there is a single attracting and a single repelling fixed point, and on the complement of the latter, iteration of the pseudo-Anosov element converges uniformly on compact sets to the former. Moreover, the set of attracting/repelling pairs for pseudo-Anosov elements are dense in $\mathcal{PMF}(S) \times \mathcal{PMF}(S)$. Therefore we have the following proposition.

\begin{prop}\label{Mod}
Let $\tau$ be the train track constructed above. If $K\subset\mathcal{PMF}(S)$ is disjoint from an open set in $\mathcal{PMF}(S)$, then there exists a mapping class $h$ such that $K\subset \mathcal{PMF}_{h(\tau)}(S)$. The same is true for $\tau'$.
\end{prop}


\section{Deformation Families of Constant Length Spectrum}\label{thm1}

We will now prove our main theorem.

\newtheorem*{theorem:main11}{Theorem~\ref{theorem:main11}} 
\begin{theorem:main11}
Let $\alpha = (\alpha_1,\dots,\alpha_{k} ; \varepsilon)$ and $(2k_0-2g+\varepsilon+1) > 0$, and suppose $\Sigma$ is a set of simple closed curves with $\overline{\Sigma}\neq\mathcal{PMF}(S)$. Then there exists a deformation family $\Omega_\Sigma\subset \mathsf{Flat}(S,\alpha)$ such that the length function $\lambda_\Sigma$ is constant on $\Omega_\Sigma$ and $\dim(\Omega_\Sigma) \geq (2k_0-2g+\varepsilon+1)$. Consequently, there exists a deformation family $\Omega_\Sigma \subset \mathsf{Flat}(S)$ of dimension at least $(6g+2n-8)$.
\end{theorem:main11}

\begin{proof}
We consider $\tau'$ constructed in Section~\ref{mftt}. Since $\overline{\Sigma} \neq \mathcal{PMF}(S)$, the set $\Sigma$ is disjoint from some open set in $\mathcal{PMF}(S)$. By Proposition~\ref{Mod}, there exists a mapping class $h_1$ such that $h_1(\tau')$ is a train track that carries $\Sigma$. By replacing $\tau'$ and $\tau$ with $h_1(\tau')$ and $h_1(\tau)$, we may assume that $\Sigma$ is carried by $\tau'$.

Denote the natural projection from $QD(S,\alpha)$ to $\mathsf{Flat}(S,\alpha)$ by 
\[
p : QD(S,\alpha) \to \mathsf{Flat}(S,\alpha).
\] 
We consider an arbitrary $\rho\in \mathsf{Flat}(S,\alpha)$ with $\hat{\rho} \in p^{-1}(\rho) \subset QD(S,\alpha)$. From Section~\ref{qdhc} we know that there exist holonomy coordinates $U_{\hat{\rho}} \to \mathbb{C}^m$ about $\hat{\rho}$.

Recall the natural map $\varphi: QD(S) \to \mathcal{MF}(S)$ and the circle of measured foliations $\operatorname{Image}(\nu_{\hat{\rho}})$ associated to $\rho$ as described in Section~\ref{mftt}. Let $K$ be the union of $\nu_{q}(\theta)$ over all $\theta \in [0,\pi)$ and $q \in U_{\hat{\rho}}$. 

Observe that by taking $U_{\hat{\rho}}$ sufficiently small, $K$ is a small neighborhood of the circle $\nu_{\hat{\rho}}([0,\pi))$. In particular $K$ can be assumed disjoint from some open set when projected to $\mathcal{PMF}(S)$. By Proposition~\ref{Mod}, there exists a mapping class $h_2$ such that the train track $h_2(\tau)$ carries $K$. Thus, by replacing $\rho$ with $h_2^{-1}(\rho)$, we may assume that $K$ is carried by $\tau$.

Use $\mathcal{MF}^{[0,\pi)}(S)$ to denote the set of functions that map $[0,\pi)$ into $\mathcal{MF}(S)$. We define
\[
f_1 : U_{\hat{\rho}} \to\mathcal{MF}^{[0,\pi)}(S)
\]
by 
\[
f_1(q) = \nu_{q}.
\]
Let $F_1$ be the image of $f_1$.

We will consider the space $W_{\tau}$, which is the space of weight functions on $\tau$ as described in Section~\ref{mftt}. Let
\[
\psi : \mathcal{MF}_{\tau}(S) \to W_{\tau}
\]
be the homeomorphism between $\mathcal{MF}_{\tau}(S)$ and $W_{\tau}$. 

Let $\sigma = \{\sigma_1,\ldots,\sigma_{6g+2n-6}\}$ be the set of curves from Proposition~\ref{Wcoord} and write $(\mathbb{R}^\sigma)^{[0,\pi)}$ for the set of functions that map $[0,\pi)$ into $\mathbb{R}^\sigma$. We define
\[
f_2 : F_1 \to (\mathbb{R}^\sigma)^{[0,\pi)} 
\]
by
\[
f_2(\nu_{q})(\theta) = \{i(\sigma_j,\nu_{q}(\theta))\}_{\sigma_j \in \sigma}.
\]
Let $F_2$ be the image of $f_2$. 

Next we let $(W_{\tau})^{[0,\pi)}$ be the space of functions that map $[0,\pi)$ into $W_{\tau}$ and define
\[
f_3 : F_2 \to (W_{\tau})^{[0,\pi)}
\]
by 
\[
f_3(h)(\theta) = A(h(\theta))
\] 
for all $h\in F_2$ and $\theta\in [0,\pi)$, where $A$ is the linear map in Proposition~\ref{Wcoord}. Equivalently, this is determined by
\[
(f_3\circ f_2)(\nu_{q})(\theta) = \psi(\nu_{q}(\theta)).
\]
Let $F_3$ be the image of $f_3$. 

For any $q \in U_{\hat{\rho}}$, the function $(f_3\circ f_2\circ f_1)(q)(\theta)$ is uniformly continuous in $\theta \in [0,\pi)$. We define
\[
f_4 : F_3 \to W_{\tau}
\]
by applying the integral $\frac{1}{2}\int_0^\pi \cdot\, d\theta$ to functions in $F_3$. We will let 
\[
f=f_4\circ f_3\circ f_2\circ f_1 : U_{\hat{\rho}} \to W_\tau.
\]
Combining everything together we have the diagram below.
\[
\begin{tikzpicture}[scale=2]
 \node (A) at (0,1) {$U_{\hat{\rho}}$};
 \node (B) at (1,1) {$F_1$};
 \node (C) at (2,1) {$F_2$};
 \node (D) at (3,1) {$F_3$};
 \node (E) at (4,1) {$W_{\tau}$};
 \node (F) at (1,.5) {$\mathcal{MF}^{[0,\pi)}(S)$};
 \node (G) at (2,.5) {$(\mathbb{R}^\sigma)^{[0,\pi)}$};
 \node (H) at (3,.5) {$(W_{\tau})^{[0,\pi)}$};
 \node (I) at (1,.75) {$\bigcap$};
 \node (J) at (2,.75) {$\bigcap$};
 \node (K) at (3,.75) {$\bigcap$};
\path[->,font=\scriptsize,>=angle 90]
(A) edge node[above]{$f_1$} (B)
(B) edge node[above]{$f_2$} (C)
(C) edge node[above]{$f_3$} (D)
(D) edge node[above]{$f_4$} (E);
\end{tikzpicture}
\]

If we define
\[
f_5 : F_2 \to \mathbb{R}^\sigma
\]
by applying the integral $\frac{1}{2}\int_0^\pi \cdot\, d\theta$ to functions in $F_2$, then by Lemma~\ref{ellq} we have
\[
(\lambda_\sigma \circ p)(q) = \left\{\frac{1}{2}\int_0^\pi i(\sigma_j ,\nu_{q}(\theta))  d\theta \right\}_{\sigma_j \in \sigma} = (f_5 \circ f_2 \circ f_1) (q),
\]
where $\lambda_\sigma$ is the length function and $p$ is the projection from quadratic differentials to flat metrics. We obtain the following commutative diagram.
\[
\begin{tikzpicture}[scale=2]
 \node (A) at (0,1) {$U_{\hat{\rho}}$};
 \node (B) at (1,1) {$F_1$};
 \node (C) at (2,1) {$F_2$};
 \node (D) at (3,1) {$F_3$};
 \node (E) at (2,.25) {$\mathbb{R}^\sigma$};
 \node (F) at (3,.25) {$W_{\tau}$};
\path[->,font=\scriptsize,>=angle 90]
(A) edge node[above]{$f_1$} (B)
(B) edge node[above]{$f_2$} (C)
(C) edge node[above]{$f_3$} (D)
(C) edge node[right]{$f_5$} (E)
(D) edge node[right]{$f_4$} (F)
(A) edge node[below]{$\lambda_\sigma \circ p$} (E)
(E) edge node[above]{$A$} (F);
\end{tikzpicture}
\]
Therefore, we have $f = A\circ \lambda_\sigma \circ p$, where $A$ is the linear map in Proposition~\ref{Wcoord}. 

Recall that $\tau$ and $\tau'$ meet efficiently. For every $c \in \mathcal{S}(S)$ carried by $\tau'$ and $\mu \in \mathcal{MF}(S)$ carried by $\tau$, $i(c,\mu) = i(c,w_\mu)$ where $w_\mu$ is the weight on $\tau$ determined by $\mu$. Appealing to Lemma~\ref{ellq}, we have
\begin{align*}
\ell(\gamma, p(q)) &= \frac{1}{2}\int_0^\pi i(\gamma,\nu_q(\theta)) d\theta \\
&= \frac{1}{2}\int_0^\pi i(\gamma,w_{\nu_q(\theta)}) d\theta \\
&= \frac{1}{2}\int_0^\pi i(\gamma,(f_3\circ f_2\circ f_1)(q)) d\theta
\end{align*}
for any $\gamma \in \Sigma$ and $q \in U_{\hat{\rho}}$. We use the fact that the intersection number of $\gamma$ with any $w_\mu$ is linear in $W_\tau$. Hence
\begin{align*}
\ell(\gamma, p(q)) &= \frac{1}{2}\int_0^\pi i(\gamma,(f_3\circ f_2\circ f_1)(q)) d\theta\\
&=  i(\gamma,(f_4\circ f_3\circ f_2\circ f_1)(q)) \\
&= i(\gamma,f(q)).
\end{align*}
Therefore the length of $\gamma \in \Sigma$ with respect to any flat metric induced by $q \in U_{\hat{\rho}}$ is equal to the intersection number $i(\gamma,f(q))$.

From Proposition~\ref{dimcount} we know that $U_{\hat{\rho}}$ is $(4g+2k+\varepsilon-3)$-dimensional, where $k=k_0+n$ is the sum of the number of unmarked zeroes and marked points. From Section~\ref{mftt} we know that $W_{\tau}$ is a $(6g+2n-6)$-dimensional space. Applying Corollary~\ref{lengsmoo} to each $\sigma_j \in \sigma$, it follows that there exists an open set $V \subset p(U_{\hat{\rho}})$ so that $\lambda_\sigma \big|_V$ is a smooth function with respect to holonomy coordinates. Since $A$ is linear, we conclude that $f\big|_{p^{-1}(V)}$ is a smooth function. This implies that the generic fiber $\Omega \subset U_{\hat{\rho}}$ has dimension at least $(2k_0-2g+\varepsilon+3)$. On $\Omega$, $f$ is constant.

Let $\gamma\in\Sigma$, $\zeta\in p(\Omega)$, and $\hat{\zeta}\in \Omega$ be a preimage of $\zeta$. The length of $\gamma$ with respect to $\zeta$ is equal to $i(f(\hat{\zeta}),\gamma)$. The length of $\gamma$ with respect to metrics in $p(\Omega)$ is independent of the metric since $f\big|_{\Omega}$ is a constant function. Therefore $p(\Omega) \subset \mathsf{Flat}(S,\alpha)$ is a deformation family of constant $\Sigma$-length-spectrum. The dimension of $p(\Omega)$ is at least $\operatorname{dim}(\Omega)-2$, which is $(2k_0-2g+\varepsilon+1)$.

For the case of $\mathsf{Flat}(S)$, we maximize the dimension of the deformation family by maximizing $k_0$. Consider the stratum where $\alpha_1=\cdots=\alpha_{k_0}=1$ and $\alpha_{k_0+1}=\cdots=\alpha_k=-1$. The value of $k_0$ is equal to $(4g+n-4)$. Consequently, there exists a deformation family of dimension at least $(6g+2n-8)$. 
\end{proof}


\section{Local Spectral Rigidity for Closed Curves}\label{thm2}

In this section we prove that for any flat metric inside a fixed stratum, we can find a finite set of closed curves that satisfies local spectral rigidity. In fact, we will explicitly construct the locally spectally rigid set of closed curves $\Sigma$ by describing geodesic representatives of $\Sigma$.

\newtheorem*{theorem:main2}{Theorem~\ref{theorem:main2}} 
\begin{theorem:main2}
Let $\alpha = (\alpha_1,\dots,\alpha_{k} ; \varepsilon)$. For any $\rho\in \mathsf{Flat}(S,\alpha)$, there exists a set of closed curves $\Sigma\subset\mathcal{C}(S)$ such that $\Sigma$ is locally spectrally rigid at $\rho \in \mathsf{Flat}(S,\alpha)$ and $|\Sigma|\leq 15(2g+k-2)$.
\end{theorem:main2}

\begin{proof}
We take a maximal collection of saddle connections with pairwise disjoint interiors with respect to $\rho$. By \cite{Vo}, this will be a triangulation of $S$ by Euclidean triangles with each side of each triangle being a saddle connection. Locally, the set of lengths of saddle connections determines the metric, so it suffices to find curves whose lengths determine the lengths of these saddle connections.

We will now prove a lemma that allows us to construct geodesic segments, i.e., concatenations of saddle connections that satisfy the angle condition. Let $C$ be the set of cone points of $\rho$. A \emph{direction} at $a\in C$ is a geodesic segment with initial point $a$. Let $\angle_a(u_1,u_2)$ denote the smaller of the pair of angles between directions $u_1$ and $u_2$ at $a\in C$.

\begin{lemma}\label{direction}
Suppose we are given $\rho\in \mathsf{Flat}(S,\alpha)$, an initial point $a\in C$, a terminal point $b\in C$, a direction $u$ at $a$, and a direction $v$ at $b$. For any $\epsilon>0$ there exists a geodesic segment $\gamma$ from $a$ to $b$ such that 
\[
\angle_a(u,\gamma) < \epsilon \text{ and } \angle_b(v,\overline{\gamma})<\epsilon,
\]
where $\overline{\gamma}$ is $\gamma$ with orientation reversed.
\end{lemma}

\begin{proof}\label{segment}
There exists a direction $u'$ arbitrarily close to $u$ such that the ray from $a$ in $u'$ direction is minimal, that is, the ray is dense on $S$. Similarly there is a minimal ray beginning at $b$ in the a direction $v'$ arbitrarily close to $v$. We can pick $u'$ and $v'$ such that the two rays intersect infinitely many times.

Consider the curve that starts at $a$ and follows along the ray for a long time before hitting an intersection and following the other ray backwards for a long time before reaching $b$. Then the geodesic segment obtained by straightening satisfies the statement as long as we use an intersection that is far enough from both $a$ and $b$. Figure~\ref{univ} shows how the curve would look in the universal cover.

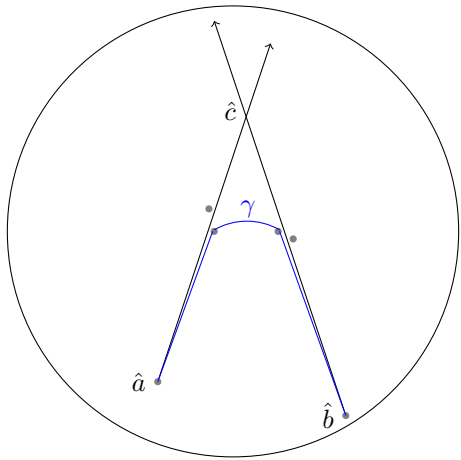
\begin{figure}[ht]
\begin{tikzpicture}
\draw (0,0) circle (3);
\draw [->] (-1,-2) node[left=1pt]{$\hat{a}$} -- (0.5,2.5);
\draw [->] (1.5,-2.45) node[left=1pt]{$\hat{b}$}-- (-0.25,2.8);
\draw (0.2,1.6) node[left=1pt]{$\hat{c}$};
\draw [blue] (0.2,0.3) node{$\gamma$};
\filldraw [gray] (-1,-2) circle (0.04);
\filldraw [gray] (1.5,-2.45) circle (0.04);
\filldraw [gray] (-0.25,0) circle (0.04);
\filldraw [gray] (-0.32,0.3) circle (0.04);
\filldraw [gray] (0.6,0) circle (0.04);
\filldraw [gray] (0.8,-0.1) circle (0.04);
\draw [blue] (-1,-2) -- (-0.27,0);
\draw [blue] (1.5,-2.45) -- (0.62,0);
\draw [blue] (-0.25,0.02) arc (120:60:.87);
\end{tikzpicture}
\caption{Visualizing the geodesic segment in the universal cover.}
\label{univ}
\end{figure}

On each side of the ray based at $a$, there will eventually be cone points at distance less than $\tan(\epsilon)$ from the ray, projecting to the ray at distance at least one from $a$. This means that these cone points are in a direction making angle less than $\epsilon$ with the ray. The same is true for the ray based at $b$. We pick an intersection point from the infinite set of intersections that is past four such cone points, i.e., so that the segments $\overline{ac}$ and $\overline{bc}$ have these nearby cone points on both sides.

Let $\hat{a}$, $\hat{b}$, and $\hat{c}$ in the universal cover of $S$ be lifts of $a$, $b$, and $c$ on $S$. We consider the geodesic triangle with vertices $\hat{a}$, $\hat{b}$, and $\hat{c}$ with the three sides being the two rays and the geodesic $\gamma$ connecting $\hat{a}$ and $\hat{b}$. By a Gauss-Bonnet Theorem argument (see \cite{R}*{Theorem~3.3}), the geodesic triangle does not contain the lift of any point in $C_0$ in the interior, where $C_0$ is the set of unmarked zeroes. The same is true for the marked points $P$ since $\gamma$ is homotopic to the concatenation of the other two sides of the triangle.

Therefore the four cone points that are of distance less than $\epsilon$ from the rays will act as barriers as we straighten to obtain $\gamma$. Hence the resulting geodesic segment $\gamma$ must satisfy the requirement: 
\[
\angle_a(u,\gamma) < \epsilon \text{ and } \angle_b(v,\overline{\gamma})<\epsilon,
\]
where $\overline{\gamma}$ is $\gamma$ with orientation reversed.
\end{proof}

Now we will describe how the length of a saddle connection can be determined in various cases. Let $\gamma_0$ be an oriented saddle connection from the triangulation and let its initial and terminal points be $a$ and $b$.

\noindent\textit{Case 1}: $a = b$ and $a$ is not a marked point. 

This means that the saddle connection itself is a closed geodesic. We know that when $\gamma_0$ is a closed curve, the length of $\gamma_0$ with respect to $\rho$ is well-defined. Pick an open neighborhood $U_\rho$ of $\rho$ small enough so that $\gamma_0$ as a closed curve is always a single saddle connection. This is possible since the condition is an open condition. Hence we add $\gamma_0$ to our set $\Sigma$. The length of $\gamma_0$ as a saddle connection is determined by the length of $\gamma_0$ as a closed curve.

\noindent\textit{Case 2}: $a \neq b$ and both $a$ and $b$ are not marked points. 

Fix a small $\epsilon>0$. Let $u_1,\ldots,u_4$ be directions at $a$ and $v_1,\ldots,v_4$ be directions at $b$ specified by the angles made with $\gamma_0$ (measured counterclockwise) according to the following table:
\[\begin{array} {clcl}
\mbox{direction at } a & \mbox{angle} & \mbox{vector at b } & \mbox{angle}\\
u_1 & \pi+2\epsilon & v_1 & -\pi-2\epsilon\\
u_2 & -\pi-2\epsilon & v_2 & \pi+2\epsilon\\
u_3 & \pi-4\epsilon & v_3 & -\pi+4\epsilon\\
u_4 & -\pi+4\epsilon & v_4 & \pi-4\epsilon\\
\end{array}\]

We apply Lemma~\ref{direction} to each pair $u_j$ and $v_j$ for $j=1,2,3,4$. We obtain geodesic segments $\gamma_j$ for $j=1,2,3,4$. Figure~\ref{case2} shows how everything fits together. 

\begin{figure}[ht]
\begin{tikzpicture}[scale=2]
 \draw (-1,0) node[below=2pt,fill=white]{$a$} node[above=18pt,fill=white]{$\pi$} node[below=16pt,fill=white]{$\pi$} node[left=20pt,fill=white]{$\geq\pi$  $\vdots$} -- node[midway]{$\gamma_0$} (1,0) node[below=2pt,fill=white]{$b$} node[above=18pt,fill=white]{$\pi$} node[below=16pt,fill=white]{$\pi$} node[right=20pt,fill=white]{$\vdots$  $\geq\pi$}; 
 \draw [<-] (-2,1) -- (-1,0); 
 \draw [<-] (-2,-1) -- (-1,0); 
 \draw [->] (1,0) -- (2,1); 
 \draw [->] (1,0) -- (2,-1); 
 \draw[red] (-1.8,1) node[right=2pt]{$\gamma_3$} -- (-1,0); 
 \draw[red] (-1.8,-1) node[right=2pt]{$\gamma_4$} -- (-1,0); 
 \draw[red] (1,0) -- (1.8,1) node[left=2pt]{$\gamma_3$}; 
 \draw[red] (1,0) -- (1.8,-1) node[left=2pt]{$\gamma_4$}; 
 \draw[blue] (-2.2,1) node[left=2pt]{$\gamma_1$} -- (-1,0); 
 \draw[blue] (-2.2,-1) node[left=2pt]{$\gamma_2$} -- (-1,0); 
 \draw[blue] (1,0) -- (2.2,1) node[right=2pt]{$\gamma_1$}; 
 \draw[blue] (1,0) -- (2.2,-1) node[right=2pt]{$\gamma_2$}; 

 \draw[red][<-] (-1.35,.5) node[right=1pt]{$u_3$} -- (-1,0); 
 \draw[red][<-] (-1.35,-.5) node[right=1pt]{$u_4$} -- (-1,0); 
 \draw[red][->] (1,0) -- (1.35,.5) node[left=1pt]{$v_3$}; 
 \draw[red][->] (1,0) -- (1.35,-.5) node[left=1pt]{$v_4$}; 
 \draw[blue][<-] (-1.65,.5) node[left=1pt]{$u_1$} -- (-1,0); 
 \draw[blue][<-] (-1.65,-.5) node[left=1pt]{$u_2$} -- (-1,0); 
 \draw[blue][->] (1,0) -- (1.65,.5) node[right=1pt]{$v_1$}; 
 \draw[blue][->] (1,0) -- (1.65,-.5) node[right=1pt]{$v_2$}; 
\draw (-1.3535,0.3535) arc (135:225:0.5);
 \draw (-0.65,0) arc (0:135:0.35);
 \draw (-0.7,0) arc (360:225:0.3);
 \draw (1.3535,0.3535) arc (45:-45:.5);
 \draw (0.65,0) arc (180:45:0.35);
 \draw (0.7,0) arc (180:315:0.3);
\end{tikzpicture}
\caption{Case 2.}
\label{case2}
\end{figure}

Add the following closed geodesics to $\Sigma$.
\[
c_1=\gamma_0\cup\gamma_1,\ c_2=\gamma_0\cup\gamma_2,\ c_3=\gamma_1\cup\gamma_4,\ c_4=\gamma_2\cup\gamma_3, \text{ and } c_5=\gamma_3\cup\gamma_4.
\]
They are all closed geodesics with respect to $\rho$ because the angle condition is satisfied by construction. Furthermore, by construction the angles between saddle connections meeting at $a$ or at $b$ are all strictly greater than $\pi$. Hence there exists an open neighborhood $U_\rho$ of $\rho$ small enough so that the angle condition at $a$ and $b$ are strictly satisfied with respect to any $\rho'$ in the neighborhood.

The length of $\gamma_0$ as a saddle connection is then described by 
\[
\frac{1}{2} [\ell(c_1,\rho') +\ell(c_2,\rho') -\ell(c_3,\rho') - \ell(c_4,\rho') + \ell(c_5,\rho')]
\]
for any $\rho'\in U_\rho$ because the geodesic representatives of $c_1$ and $c_2$ both always contain $\gamma_0$.

\noindent\textit{Case 3}: Both $a$ and $b$ are marked points. 

The argument is similar to Case 1. Consider a simple closed curve that encloses only $\gamma_0$ but no other cone points besides $a$ and $b$. The geodesic representative would be $\gamma_0\cup\overline{\gamma_0}$, where $\overline{\gamma_0}$ is $\gamma_0$ with orientation reversed. The length of $\gamma_0$ as a saddle connection is determined by the length of this closed geodesic by a factor of one half.

\noindent\textit{Case 4}: Only one of $a$ and $b$ is a marked point. 

Without loss of generality, let $b$ be the marked point. The setting is similar to Case 2 with directions specified by the angles made with $\gamma_0$ (measured counterclockwise). 
\[\begin{array} {cl}
\mbox{direction at } a & \mbox{angle}\\
u_1 & \pi+2\epsilon \\
u_2 & -\pi-2\epsilon \\
u_3 & \pi-4\epsilon \\
u_4 & -\pi+4\epsilon \\
\end{array}\]

We apply the lemma to the pair $u_1$ and $u_2$ to obtain geodesic segment $\gamma_1$. Use the pair $u_3$ and $u_4$ to obtain $\gamma_2$ as in Figure~\ref{case4}. 

\begin{figure}[ht]
\begin{tikzpicture}[scale=2]
 \draw (-1,0) node[below=2pt,fill=white]{$a$} node[above=18pt,fill=white]{$\pi$} node[below=16pt,fill=white]{$\pi$} node[left=20pt,fill=white]{$\geq\pi$  $\vdots$} -- node[midway]{$\gamma_0$} (1,0) node[below=2pt,fill=white]{$b$}; 
 \draw [<-] (-2,1) -- (-1,0); 
 \draw [<-] (-2,-1) -- (-1,0); 
 \draw[red] (-1.8,1) node[right=2pt]{$\gamma_2$} -- (-1,0); 
 \draw[red] (-1.8,-1) node[right=2pt]{$\gamma_2$} -- (-1,0); 
 \draw[blue] (-2.2,1) node[left=2pt]{$\gamma_1$} -- (-1,0); 
 \draw[blue] (-2.2,-1) node[left=2pt]{$\gamma_1$} -- (-1,0); 
 \draw[red][<-] (-1.35,.5) node[right=1pt]{$u_3$} -- (-1,0); 
 \draw[red][<-] (-1.35,-.5) node[right=1pt]{$u_4$} -- (-1,0); 
 \draw[blue][<-] (-1.65,.5) node[left=1pt]{$u_1$} -- (-1,0); 
 \draw[blue][<-] (-1.65,-.5) node[left=1pt]{$u_2$} -- (-1,0); 
  \draw (-1.3535,0.3535) arc (135:225:0.5);
 \draw (-0.65,0) arc (0:135:0.35);
 \draw (-0.7,0) arc (360:225:0.3);
 \draw (1,0) circle (.5pt); 
\end{tikzpicture}
\caption{Case 4.}
\label{case4}
\end{figure}

Add the following closed geodesics to $\Sigma$.
\[
c_1=\gamma_0\cup\gamma_1\cup\overline{\gamma_0},\ c_2=\gamma_1\cup\gamma_2, \text{ and } c_3=\gamma_2.
\]
Once again we can find $U_\rho$ where the angle condition at $a$ and at $b$ are both strict. By a similar argument, the length of $\gamma_0$ is described by 
\[
\frac{1}{2} [\ell(c_1,\rho') -\ell(c_2,\rho') + \ell(c_3,\rho')]
\]
for any $\rho'\in U_\rho$.

We constructed at most five closed geodesics for each saddle connection in the triangulation. One can compute the number of saddle connections in the triangulation, which is $(6g+3k-6)$. Hence $|\Sigma|\leq 5(6g+3k-6)$. Finally we pick $U_\rho$ that respects the triangulation and the expression of the length of each saddle connection by the lengths of closed geodesics. This means that the set of saddle connections for the triangulation form a triangulation for any $\rho'\in U_\rho$. The second part means that, for each saddle connection $\gamma_0$, the conclusion to the case analysis is true with respect to any $\rho'\in U_\rho$. Hence we are done.
\end{proof}



\begin{bibdiv}
\begin{biblist}
 
\bib{DLR}{article}{
   author={Duchin, Moon},
   author={Leininger, Christopher J.},
   author={Rafi, Kasra},
   title={Length spectra and degeneration of flat metrics},
   journal={Invent. Math.},
   volume={182},
   date={2010},
   number={2},
   pages={231--277},
   issn={0020-9910},
   review={\MR{2729268 (2011m:57022)}},
   doi={10.1007/s00222-010-0262-y},
}

\bib{FLP}{collection}{
   author={Fathi, A.},
   author={Laudenbach, F.},
   author={Po\'{e}naru, V.},
   title={Travaux de Thurston sur les surfaces},
   language={French},
   note={S\'eminaire Orsay;
   Reprint of {\it Travaux de Thurston sur les surfaces}, Soc.\ Math.\
   France, Paris, 1979 [ MR0568308 (82m:57003)];
   Ast\'erisque No. 66-67 (1991)},
   publisher={Soci\'et\'e Math\'ematique de France},
   place={Paris},
   date={1991},
   pages={1--286},
   issn={0303-1179},
   review={\MR{1134426 (92g:57001)}},
}

\bib{M}{article}{
   author={Masur, Howard},
   title={Interval exchange transformations and measured foliations},
   journal={Ann. of Math. (2)},
   volume={115},
   date={1982},
   number={1},
   pages={169--200},
   issn={0003-486X},
   review={\MR{644018 (83e:28012)}},
   doi={10.2307/1971341},
}

\bib{PH}{book}{
   author={Penner, R. C.},
   author={Harer, J. L.},
   title={Combinatorics of train tracks},
   series={Annals of Mathematics Studies},
   volume={125},
   publisher={Princeton University Press},
   place={Princeton, NJ},
   date={1992},
   pages={xii+216},
   isbn={0-691-08764-4},
   isbn={0-691-02531-2},
   review={\MR{1144770 (94b:57018)}},
}

\bib{R}{article}{
   author={Rafi, Kasra},
   title={A characterization of short curves of a Teichm\"uller geodesic},
   journal={Geom. Topol.},
   volume={9},
   date={2005},
   pages={179--202},
   issn={1465-3060},
   review={\MR{2115672 (2005i:30072)}},
   doi={10.2140/gt.2005.9.179},
}

\bib{S}{book}{
   author={Strebel, Kurt},
   title={Quadratic differentials},
   series={Ergebnisse der Mathematik und ihrer Grenzgebiete (3) [Results in
   Mathematics and Related Areas (3)]},
   volume={5},
   publisher={Springer-Verlag},
   place={Berlin},
   date={1984},
   pages={xii+184},
   isbn={3-540-13035-7},
   review={\MR{743423 (86a:30072)}},
}

\bib{T}{article}{
   author={Thurston, William P.},
   title={The Geometry and Topology of Three-manifolds},
   journal={Princeton Lecture Notes},
   date={1980},
}

\bib{Ve}{article}{
   author={Veech, William A.},
   title={The Teichm\"uller geodesic flow},
   journal={Ann. of Math. (2)},
   volume={124},
   date={1986},
   number={3},
   pages={441--530},
   issn={0003-486X},
   review={\MR{866707 (88g:58153)}},
   doi={10.2307/2007091},
}

\bib{Vo}{article}{
   author={Vorobets, Yaroslav},
   title={Periodic geodesics on generic translation surfaces},
   conference={
      title={Algebraic and topological dynamics},
   },
   book={
      series={Contemp. Math.},
      volume={385},
      publisher={Amer. Math. Soc.},
      place={Providence, RI},
   },
   date={2005},
   pages={205--258},
   review={\MR{2180238 (2007d:37048)}},
   doi={10.1090/conm/385/07199},
}

\bib{Z}{article}{
   author={Zorich, Anton},
   title={Flat surfaces},
   conference={
      title={Frontiers in number theory, physics, and geometry. I},
   },
   book={
      publisher={Springer},
      place={Berlin},
   },
   date={2006},
   pages={437--583},
   review={\MR{2261104 (2007i:37070)}},
   doi={10.1007/978-3-540-31347-2-13},
}

\end{biblist}
\end{bibdiv}
\end{document}